\documentclass[11pt,leqno]{article}
\usepackage{amsmath, amssymb, amsthm}

\theoremstyle{plain}
\newtheorem{thm}{Theorem}
\newtheorem{cor}{Corollary}
\newtheorem{lem}{Lemma}[section]

\theoremstyle{remark}
\newtheorem{rem}{Remark}

\numberwithin{equation}{section}

\title{ On the Rankin-Selberg problem: higher power moments of the Riesz mean error term }
\author{Yoshio Tanigawa(Nagoya), Wenguang Zhai and Deyu Zhang(Jinan)}
\date{}

\begin{document}

\maketitle

\begin{center}
{Science in China Series A {\bf 51}(1)(2008), 148-160. }
\end{center}

\begin{abstract}
Let $\Delta_1(x;\varphi)$ be the error term of the first Riesz means of the Rankin-Selberg
zeta function. We study the higher power moments of $\Delta_1(x;\varphi)$ and derive
an asymptotic formula for 3-rd, 4-th and 5-th power moments by using Ivi\'c\,'s large
value arguments.
\end{abstract}

\footnote[0]{This work is supported by National Natural Science Foundation of China(Grant No.
10301018) and   National Natural Science Foundation of Shandong Province(Grant No. 2006A31).}

\section{Introduction and  main results }
Let $\varphi(z)$ be a holomorphic form of weight $\kappa$ with respect to the full modular group
$SL_2({\Bbb Z})$, and denote by $a(n)$ the $n$-th Fourier coefficient of $\varphi(z).$ Suppose that
 $\varphi(z)$ is normalized such that $a(1)=1$ and $T(n)\varphi=a(n)\varphi$ for every $n\in {\Bbb N},$
where $T(n)$ is  the Hecke operator of order $n$. Let
$$
c_n=n^{1-\kappa}\sum_{m^2|n}m^{2(\kappa-1)}\left|a\left(\frac{n}{m^2}\right)\right|^2.
$$
Deligne's \cite{D} proved that  $|a(n)|\leq n^{(\kappa-1)/2}d(n),$
which  implies  $c_n\ll n^\varepsilon,$ where $d(n)$ is the Dirichlet divisor function.
The classical Rankin-Selberg problem is to study the upper bound of the error term
\begin{equation}
\Delta(x;\varphi):=\sum_{n\leq x}c_n-Cx,
\end{equation}
where $C$ is an explicit constant. Rankin \cite{R} proved that
\begin{equation}
\Delta(x;\varphi)=O(x^{\frac 35}),
\end{equation}
which was stated by Selberg \cite{S} again without proof. However, no improvement of (1.2) has been
obtained after Rankin and Selberg. Ivi\'c \cite{I1} proved that
$\Delta(x;\varphi)=\Omega_{\pm}(x^{\frac 38})$ and conjectured that
$\Delta(x;\varphi)=O(x^{\frac 38+\varepsilon})$.

In \cite{IMT}, Ivi\'c, Matsumoto and Tanigawa studied  the Riesz mean of the type
$$
D_{\rho}(x;\varphi):=\frac{1}{\Gamma(\rho+1)}\sum_{n\leq x}(x-n)^\rho c_n.
$$
Define the error term $\Delta_\rho(x;\varphi)$ by
\begin{equation}
D_\rho(x;\varphi)=\frac{\pi^2 \kappa
R_0}{6\Gamma(\rho+2)}x^{\rho+1}+\frac{Z(0)}{\Gamma(\rho+1)}x^\rho+\Delta_\rho(x;\varphi),
\end{equation}
where
\begin{align*}
R_0&=\frac{12(4\pi)^{\kappa-1}}{\Gamma(\kappa+1)}\iint_{\frak{F}}y^{\kappa-2}|\varphi(z)|^2dxdy,\\
Z(s)&=\sum_{n=1}^\infty c_n n^{-s},\ \ \Re s>1,
\end{align*}
the integral being taken over a fundamental domain $\frak{F}$ of
$SL_2({\Bbb Z})$ and $Z(s)$ can be continued to the whole plane.

Ivi\'c, Matsumoto and Tanigawa \cite{IMT} studied the relation between $\Delta(x;\varphi)$
and $\Delta_1(x;\varphi)$. They proved if $\Delta_1(x;\varphi)=O(x^\alpha)$ holds for some
$\alpha\geq 0$, then $\Delta(x;\varphi)=O(x^{\alpha/2})$. They also proved that
\begin{equation}
\Delta_1(x;\varphi)=O(x^{\frac 65})
\end{equation}
 and
\begin{equation}
\int_1^T\Delta_1^2(x;\varphi)dx=\frac{2}{13}(2\pi)^{-4}(\sum_{n=1}^\infty
c_n^2 n^{-7/4} )T^{13/4}+O(T^{3+\varepsilon}).
\end{equation}
Recently, Ivi\'c \cite{I2} studied the fourth power moment of $\Delta_1(x;\varphi)$ and  proved that
\begin{equation}
\int_1^T\Delta_1^4(x;\varphi)dx \ll T^{11/2+\varepsilon}.
\end{equation}
Note that (1.6) is the best possible up to $\varepsilon$. Both of
(1.5) and (1.6) support the following

\noindent {\bf Conjecture.} The pointwise estimate
\begin{equation}
\Delta_1(x;\varphi) \ll x^{9/8+\varepsilon}
\end{equation}
holds.

In this paper we shall prove that for the fourth-power moment of $\Delta_1(x;\varphi)$,
we can get an asymptotic formula, which substantially improves Ivi\'c\,'s estimate (1.6).
More generally, we shall establish the asymptotic formula for the
$k$-th power moment of $\Delta_1(x;\varphi)$ with $k=3,4,5.$  We first prove the
following Theorem 1, which is an upper bound result.

\begin{thm} Suppose $A\geq 0$ is a fixed constant,  then we
have
\begin{equation}
\int_T^{2T}|\Delta_1(x;\varphi)|^Adx
\ll\left\{\begin{array}{ll} T^{1+ 9A/8}\log^{39} T,&\mbox{if $0\leq A\leq 16/3 ,$}\\[1ex]
                            T^{(3+6A)/5 }\log^{5} T,& \mbox{if $A>16/3.$}
          \end{array}\right.
\end{equation}
\end{thm}

\begin{rem} From Theorem 1 we know that the estimate
\begin{equation}
\int_T^{2T}|\Delta_1(x;\varphi)|^{A_0}dx\ll T^{1+9A_0/8+\varepsilon}
\end{equation}
holds for $A_0=16/3.$ The value of $A_0$ for which (1.9) holds is closely related to the upper
bound estimate of $\Delta_1(x;\varphi).$ The exponent $16/3$ follows from the estimate (1.4).
\end{rem}

Before stating the asymptotic results, we introduce some notations. Suppose
$f:{\Bbb N}\rightarrow {\Bbb R}$ is any function such that $f(n)\ll n^\varepsilon,$ $k\geq 2$
is a fixed integer. Define
\begin{equation}
s_{k;l}(f):=\sum_{\sqrt[4]{n_1}+\cdots
+\sqrt[4]{n_l}=\sqrt[4]{n_{l+1}}+\cdots +\sqrt[4]{n_k }}
\frac{f(n_1)\cdots f(n_k)}{(n_1\cdots n_k)^{7/8}}\hspace{3mm}(1\leq
l<k),
\end{equation}
\begin{equation}
B_k(f):=\sum_{l=1}^{k-1}{k-1 \choose l}s_{k;l}(f)\cos\frac{\pi(k-2l)}{4}.
\end{equation}
We shall use $s_{k;l}(f)$  to denote both of the series (1.10) and
its value.

Suppose $A_0>3$ is a real number, define
\begin{align*}
&K_0: =\min\{n\in {\Bbb N}:n\geq A_0, 2|n\},\ \ b(k):=4^{k-2}+\frac{k-2}{8},\\
&\sigma(k,A_0): = \frac{3(A_0-k)}{4(A_0-2)}, \ \  3\leq k<A_0,  \\
&\delta_1(k,A_0): =\frac{\sigma(k,A_0)}{4b(K_0)},\ \
 \delta_2(k,A_0): =\frac{\sigma(k,A_0)}{4b(k)+4\sigma(k,A_0)}.
\end{align*}

\begin{thm} Let $A_0>5$ be a real number such that (1.9) holds, then for any integer
$3\leq k<A_0$, we have the asymptotic formula
\begin{equation}
\int_1^T\Delta_1^k(x;\varphi)dx=\frac{B_k(c)}{(8+9k)2^{3k-4}\pi^{2k}}T^{1+9k/8}
+O(T^{1+9k/8-\delta_1(k,A_0)+\varepsilon}).
\end{equation}
\end{thm}

\begin{rem}  From Theorem 1 we see that  $k\in \{3,4,5\},$ we
can get the asymptotic formula (1.12). Moreover, if the estimate
$\Delta_1(x;\varphi)\ll x^{19/16-\delta}$ holds for some small
$\delta>0,$ then the asymptotic formula (1.12) holds for $k =6.$
\end{rem}

The constant $\delta_1(k,A_0)$ is very small when $k=3,4,5.$ In
these three cases, we have the following Theorem 3.

\begin{thm} If $k=3,4,5,$ then
\begin{equation}
\int_1^T\Delta_1^k(x;\varphi)dx=\frac{B_k(c)}{(8+9k)2^{3k-4}\pi^{2k}}T^{1+9k/8}+O(T^{1+9k/8-\delta_2(k,16/3)+\varepsilon}).
\end{equation}
\end{thm}

\begin{rem} Note that
\begin{eqnarray*}
&&\delta_2(3,16/3)=7/248>1/36,\ \ \delta_2(4,16/3)=3/662>1/221,\\
&&\delta_2(5,16/3)=3/5192>1/1731.
\end{eqnarray*}
So we get the following corollary.
\end{rem}

\begin{cor} We have the following asymptotic formulas,
\begin{align*}
&\int_1^T\Delta_1^3(x;\varphi)dx=\frac{B_3(c)}{1120\pi^{6}}T^{35/8}+O(T^{35/8-1/36+\varepsilon}),\\
&\int_1^T\Delta_1^4(x;\varphi)dx=\frac{B_4(c)}{11264\pi^{8}}T^{11/2}+O(T^{11/2-1/221+\varepsilon}),\\
&\int_1^T\Delta_1^5(x;\varphi)dx=\frac{B_5(c)}{108544\pi^{10}}T^{53/8}+O(T^{53/8-1/1731+\varepsilon}).
\end{align*}
\end{cor}

%===========================================================================
\section{\bf Preliminary Lemmas}

\begin{lem} Suppose $x>1$ is a  real number . For any $N\geq 1,$ define
\begin{equation}
 {\cal R}_1(x;N)=(2\pi)^{-2}x^{9/8}\sum_{n\leq
N}\frac{c_n}{n^{7/8}}\cos\left(8\pi\sqrt[4]{nx}-\frac \pi 4\right).
\end{equation}
If $1\ll N\ll x^2,$ then
\begin{equation}
\Delta_1(x;\varphi)={\cal
R}_1(x;N)+O(x^{1+\varepsilon}+x^{3/2+\varepsilon}N^{-1/2}).
\end{equation}
\end{lem}

\begin{proof}(2.2) is a special case of Lemma 2 of \cite{IMT} .\end{proof}

\begin{lem} Let $S$ be an inner-product vector space over
${\Bbb C},$ $(a,b)$ denote the inner product in $S$ and $\Vert
a\Vert^2=(a,a).$ Suppose that $\xi,\varphi_1,\cdots,\varphi_R$ are
arbitrary vectors in $S$. Then
$$\sum_{l\leq R}|(\xi,\varphi_l)|^2\leq \Vert
\xi\Vert^2 \max_{l_1\leq R}\sum_{l_2\leq
R}|(\varphi_{l_1},\varphi_{l_2})|.$$
\end{lem}

\begin{proof} This is the well-known Halasz-Montgomery inequality. See
(A.30) of Ivi\'c \cite{I3}.\end{proof}

\begin{lem} Suppose $k\geq 3,$  $(i_1,\cdots, i_{k-1})\in \{0,1\}^{k-1}$ such that
$$
\sqrt[4]{n_1}+(-1)^{i_1}\sqrt[4]{n_2}+(-1)^{i_2}\sqrt[4]{n_3}+\cdots
+(-1)^{i_{k-1}}\sqrt[4]{n_k} \not= 0.
$$
Then we have
$$
|\sqrt[4]{n_1}+(-1)^{i_1}\sqrt[4]{n_2}+(-1)^{i_2}\sqrt[4]{n_3}+\cdots
+(-1)^{i_{k-1}}\sqrt[4]{n_k}| \gg \max(n_1,\cdots,
n_k)^{-(4^{k-2}-4^{-1}) }.
$$
\end{lem}

\begin{proof} We follow Tsang's approach \cite{T}. Consider the number field
${\Bbb Q}(\sqrt[4]{n_1},\cdots, \sqrt[4]{n_k}).$ Its degree $h$ is  one of the following numbers:
$1, 4, 4^2,\cdots, 4^k.$ If $h=1,$ it is trivial.
Suppose $h\geq 4.$ Let
$$
\alpha=\sqrt[4]{n_1}+(-1)^{i_1}\sqrt[4]{n_2}+(-1)^{i_2}\sqrt[4]{n_3}+\cdots
+(-1)^{i_{k-1}}\sqrt[4]{n_k}
$$
and $\alpha=\alpha_1, \alpha_2,
\alpha_3,\cdots, \alpha_h$ denotes all conjugates of $\alpha.$
Without loss of generality, suppose
\begin{align*}
\alpha_{\frac{h}{4}+j}&=-\alpha_j,\ \ 1\leq j\leq h/4; \\
\alpha_{\frac{h}{2}+j}&=\sqrt{-1}\alpha_j,\ \ 1\leq j\leq h/4,\\
\alpha_{\frac{3h}{4}+j}&=-\sqrt{-1}\alpha_j,\ \ 1\leq j\leq h/4.
\end{align*}
Since $\alpha_1\alpha_2\cdots \alpha_h$ is a nonzero integer, we
have $|\alpha_1\alpha_2\cdots \alpha_h|\geq 1$ and hence
$$
|\alpha_1\alpha_2\cdots \alpha_{h/4}|\geq 1,
$$
which combining the trivial estimate  $|\alpha_j|\leq k\max(n_1, \cdots, n_k)$ gives
\begin{align*}
|\alpha|&\geq \frac{1}{\prod_{j=2}^{h/4}|\alpha_j|}\geq
k^{1-h/4}\max(n_1,\cdots, n_k)^{-(h/16-1/4)}\\
&\geq k^{1-4^{k-1}}\max(n_1,\cdots, n_k)^{-(4^{k-2}-1/4)}
\end{align*}
if noting that $h\leq 4^k.$ This completes the proof of Lemma 2.3.
 \end{proof}

\begin{lem} Suppose $y>1.$ Define
\begin{equation*}
s_{k;l}(f;y):=\sum_{\stackrel{\sqrt[4]{n_1}+\cdots
+\sqrt[4]{n_l}=\sqrt[4]{n_{l+1}}+\cdots +\sqrt[4]{n_k }}{n_j\leq y,\
1\leq j\leq k}} \frac{f(n_1)\cdots f(n_k)}{(n_1\cdots
n_k)^{7/8}}\hspace{3mm}(1\leq l<k).
\end{equation*}
Then
$$|s_{k;l}(f;y)-s_{k;l}(f)|\ll y^{-3/4+\varepsilon}.$$
\end{lem}

\begin{proof}
This lemma can be proved by the same argument of Lemma 3.1 of the
second author \cite{Z}, so we omit the details.
\end{proof}

%==========================================================================

\section{\bf Proof of Theorem 1}
In order to prove Theorem 1, it suffices for us to  study the upper bound of the integral
$\int_T^{2T}|\Delta_1(x;\varphi)|^Adx $.

%-----------------------------------------------------------------------------
\subsection{\bf A large value estimate of $\Delta_1(x;\varphi)$}
We first prove the following Theorem 4 via Ivi\'c approach (see,
Chapter 13 of \cite{I3}).

\begin{thm} Suppose $T\leq x_1<x_2<\cdots <x_M\leq 2T$
satisfies $|\Delta_1(x_l;\varphi)|\gg VT \ \ (l=1,2,\cdots,M)$ and
$|x_j-x_i|\gg V\gg T^{2/21}{\cal L}^{34/21} \, (i\not= j),$ then we have
\begin{equation}
M \ll TV^{-3}{\cal L}^5+T^{4}V^{-25}{\cal L}^{39},
\end{equation}
where we put ${\cal L}=\log T$.
\end{thm}

\begin{proof} Suppose $V<T_0$ is a parameter to be determined later.
Let $I$ be any subinterval of $[T,2T]$ of length not exceeding $T_0$
and let $G=I\cap \{x_1,x_2,\cdots,x_M\}$. Without loss of
generality, we may suppose $G=\{x_1,x_2,\cdots,x_{M_0}\}$.

Suppose $J=\left[\frac{(1+4\varepsilon){\cal L}-2\log V}{\log 2}\right]$,
then $N=2^{J+1} \asymp T^{1+4\varepsilon}V^{-2}$. By Lemma 2.1 we get ($T\leq x\leq
2T$)
\begin{align*}
\Delta_1(x;\varphi)
& \ll T^{\frac 98}\left|\sum_{n\leq N}\frac{c_n}{n^{7/8}}e(4(nx)^{1/4})\right|+T^{1-\varepsilon}V \\
& \ll T^{\frac 98}\left|\sum_{j=0}^{J} \sum_{n \sim 2^j} \frac{c_n}{n^{7/8}}e(4(nx)^{1/4})\right|
      +T^{1-\varepsilon}V.
\end{align*}
Squaring, summing over the set $G$ and then using the Cauchy inequality we get that
\begin{align}
\sum_{l \leq M_0} \Delta_1^2(x_l;\varphi)
& \ll T^{\frac 94} \sum_{l \leq M_0} \left|\sum_{j}\sum_{n\sim 2^{j}}
       c_n n^{-7/8}e\left(4(nx_l)^{1/4}\right)\right|^2 \\
& \ll T^{\frac 94}{\cal L}\sum_{l \leq M_0}\sum_{j}\left|\sum_{n\sim 2^{j}}c_n
         n^{-7/8}e\left(4(nx_l)^{1/4}\right)\right|^2 \nonumber \\
& \ll T^{\frac 94}{\cal L}^2 \sum_{l \leq M_0} \left|\sum_{n\sim 2^{j_0}}
         c_nn^{-7/8}e\left(4(nx_l)^{1/4}\right)\right|^2 \nonumber
\end{align}
for some $0\leq j_0 \leq J$. Let $N_0=2^{j_0}$.

Take $\xi=\{\xi_n\}_{n=1}^{\infty}$ with $\xi_n=c_nn^{-7/8}$ for
$n\sim N_0$ and zero otherwise, and take
$\varphi_l=\{\varphi_{l,n}\}_{n=1}^{\infty}$ with
$\varphi_{l,n}=e(4(nx_l)^{1/4})$ for $n\sim N_0$ and zero otherwise.
Then
\begin{align*}
&(\xi,\varphi_l)=\sum_{n\sim N_0}\frac{c_n}{n^{7/8}}e(4(nx_l)^{1/4}),\\
&(\varphi_{l_1},\varphi_{l_2})=\sum_{n\sim N_0}e\left(4n^{1/4}(x_{l_1}^{1/4}-x_{l_2}^{1/4})\right),\\
& \Vert \xi\Vert^2=\sum_{n\sim N_0}\frac{c_n^2}{n^{7/4}}\ll
  N_0^{-7/4}\sum_{n\sim N_0}c_n^2\ll N_0^{-3/4}{\cal L}^{1+\varepsilon},
\end{align*}
where we used the bound
\begin{equation}
\sum_{n\leq x}c_n^2\ll x\log^{1+\varepsilon} x
\end{equation}
proved in \cite{M}.

From (3.2) and Lemma 2.2  we get
\begin{align}
M_0V^2 & \ll T^{1/4}{\cal L}^2\sum_{l \leq M_0}\left|\sum_{n\sim N_0}
  c_n n^{-7/8}e\left(4(nx_l)^{1/4}\right)\right|^2\\
& \ll {\cal L}^{3+\varepsilon}T^{1/4}N_0^{-3/4}\max_{l_1\leq M_0}\sum_{l_2\leq
       M_0}\left|\sum_{n\sim N_0}e\left(4n^{1/4}(x_{l_1}^{1/4}-x_{l_2}^{1/4})\right) \right|\nonumber\\
& \ll {\cal L}^{3+\varepsilon}T^{1/4}N_0^{1/4}+\frac{{\cal L}^{3+\varepsilon}T^{1/4}}{N_0^{3/4}}
       \max_{l_1\leq M_0}\sum_{\begin{subarray}{c} l_2\leq M_0, \\ l_2\not= l_1 \end{subarray}}
       \left|\sum_{n\sim N_0}e\left(4n^{1/4}(x_{l_1}^{1/4}-x_{l_2}^{1/4})\right) \right|.\nonumber
\end{align}

By the Kuzmin-Landau inequality and the exponent pair
$$
(1/11,9/11)=C_{6/11}((0,1),(1/6,4/6))
$$
we get
\begin{align*}
\sum_{n\sim N_0}e\left(4n^{1/4}(x_{l_1}^{1/4}-x_{l_2}^{1/4})\right)
& \ll \frac{N_0^{3/4}}{|\sqrt[4]{x_{l_1}}-\sqrt[4]{x_{l_2}}|}
      +\left(\frac{|\sqrt[4]{x_{l_1}}-\sqrt[4]{x_{l_2}}|}{N_0^{3/4}}\right)^{1/11}N_0^{9/11}\\
& \ll \frac{(N_0T)^{3/4}}{|x_{l_1}-x_{l_2}|}
      +\left(\frac{|x_{l_1}-x_{l_2}|}{(N_0T)^{3/4}}\right)^{1/11}N_0^{9/11}\\
& \ll \frac{(N_0T)^{3/4}}{|x_{l_1}-x_{l_2}|} +T^{-3/44}T_0^{1/11}N_0^{3/4},
\end{align*}
where we used the mean value theorem and the estimate $|x_{l_1}-x_{l_2}|\leq T_0.$

Substituting this estimate into (3.4) we get
\begin{align}
M_0V^2
&\ll {\cal L}^{3+\varepsilon}T^{1/4}N_0^{1/4}+{\cal L}^{3+\varepsilon}T^{1/4}N_0^{-3/4}\\
& \hspace{15mm}\times \max_{l_1\leq M_0}
  \sum_{\begin{subarray}{c} l_2\leq M_0,\\ l_2\not=l_1 \end{subarray}}
  \left(\frac{(N_0T)^{3/4}}{|x_{l_1}-x_{l_2}|}+T^{-3/44}T_0^{1/11}N_0^{3/4} \right)\nonumber\\
&\ll {\cal L}^{3+\varepsilon}T^{1/4}N_0^{1/4}+{\cal L}^{4+\varepsilon}TV^{-1}+{\cal L}^{3+\varepsilon}
     M_0T^{2/11}T_0^{1/11}\nonumber\\
&\ll  {\cal L}^{4+\varepsilon}TV^{-1}+{\cal L}^{3+\varepsilon}M_0T^{2/11}T_0^{1/11},\nonumber
\end{align}
where we used the facts that $\{x_l\}$ is $V$-spaced and $N_0\ll T^{1+4\varepsilon}V^{-2}$.
Take $T_0=V^{22}T^{-2}{\cal L}^{-34}$, it is easy to check that $T_0\gg V$ if
$V\gg T^{2/21}{\cal L}^{34/21}.$ We get for this $T_0$ that
$$
M_0\ll {\cal L}^{4+\varepsilon}TV^{-3}.
$$

Now we divide the interval  $[T,2T]$ into  $O(1+T/T_0)$ subintervals
of length not exceeding $T_0.$ In each interval of this type, the
number of $x_l$'s is  at most $O({\cal L}^{4+\varepsilon}TV^{-3})$ .
So we have
\begin{equation}
M\ll {\cal L}^{4+\varepsilon}TV^{-3}(1+\frac{T}{T_0})\ll {\cal L}^5TV^{-3}+{\cal L}^{39}T^{4}V^{-25}.
\end{equation}

This completes the proof of Theorem 4.
\end{proof}

%---------------------------------------------------------------
\subsection{\bf Proof of Theorem 1}

Now we prove Theorem 1. When $A=0$, Theorem 1 is trivial . When $0<A<2$, it follows from (1.6)
and the H\"older's inequality. So later we always suppose $A>2$.  It suffices for us to prove the
estimate
\begin{equation}
\int_T^{2T}|\Delta_1(x;\varphi)|^Adx
\ll\left\{\begin{array}{ll} T^{1+ 9A/8}\log^{39} T,&\mbox{if $2<A\leq 16/3,$}\\[1ex]
                            T^{(3+6A)/5 }\log^{5} T,& \mbox{if $A>16/3.$}
\end{array}\right.
\end{equation}

Suppose $x^\varepsilon\ll y\leq x/2.$ By the definition of $D_\rho(x;\varphi)$ we get
$$
\Delta_1(x;\varphi)= \sum_{n\leq x}(x-n)  c_n  -\frac{\pi^2 \kappa R_0}{12}x^{2}- Z(0) x.
$$
So we have
\begin{align*}
& |\Delta_1(x+y;\varphi)-\Delta_1(x;\varphi)|\\
& \ll |\sum_{n\leq x+y}(x+y-n)  c_n-\sum_{n\leq x}(x-n) c_n|+|(x+y)^2-x^2|+y\\
& \ll y|\sum_{n\leq x} c_n |+|\sum_{x<n\leq x+y}(x+y-n) c_n|+yx\ll yx,
\end{align*}
where we used the bound $\sum_{x<n\leq x+y}c_n\ll y$ proved in \cite{IMT}.
So there exists an absolute constant $c_0>0$ such that
$$
|\Delta_1(x+y;\varphi)-\Delta_1(x;\varphi)|\leq c_0xy,
$$
which implies that if $|\Delta_1(x;\varphi)|\geq 2c_0xy$, then
\begin{equation}
|\Delta_1(x+y;\varphi)|\geq |\Delta_1(x;\varphi)|-|\Delta_1(x+y;\varphi)-\Delta_1(x;\varphi)|
\geq c_0xy.
\end{equation}

From (3.8) and an argument  similar to (13.70) of Ivi\'c [4] we may write
\begin{equation}
\int_T^{2T}|\Delta_1(x;\varphi)|^Adx\ll T^{\frac{8+9A}{8}}
+\sum_{V}V\sum_{r\leq N_V}|\Delta_1(x_r;\varphi)|^A,
\end{equation}
where $T^{1/8}\leq V=2^m\leq T^{1/5}$, $VT<|\Delta_1(x_r;\varphi)|\leq 2VT$ $(r=1,\cdots, N_V)$
and $|x_r-x_s|\geq V$ for $r\not= s\leq N=N_V.$

Suppose first $2<A< 24$.   By Theorem 4  we get
\begin{align}
& V \sum_{r\leq N_V}|\Delta_1(x_r;\varphi)|^A\ll T^{A}N_VV^{A+1}\\
& \ll {\cal L}^5T^{1+A}V^{A-2}+{\cal L}^{39}T^{4+A}V^{A-24}\nonumber\\
& \ll T^{1+A+\frac{1}{5}(A-2)}{\cal L}^{5}+ T^{1+\frac{9A}{8} }{\cal L}^{39}. \nonumber
\end{align}
If  $A> 24,$  then by Theorem 4 again we get
\begin{align}
& V\sum_{r\leq N_V}|\Delta_1(x_r;\varphi)|^A\ll T^{A}N_VV^{A+1}\\
& \ll {\cal L}^5T^{1+A}V^{A-2}+{\cal L}^{39}T^{4+A}V^{A-24}\nonumber\\
& \ll T^{1+A+\frac{1}{5}(A-2)}{\cal L}^{5}+ T^{ \frac{6A-4}{5}}{\cal L}^{39}\nonumber\\
& \ll  T^{ \frac{6A+3}{5}}{\cal L}^{5}.\nonumber
\end{align}
From (3.8)--(3.11) we get (3.7) easily. This completes the proof of Theorem 1.

%============================================================================
\section{\bf Proofs of Theorem 2 and Theorem 3}

Suppose $T\geq 10$ is a real number. It suffices for us to evaluate the integral
$\int_T^{2T}\Delta_1^k(x)dx$. Suppose $y$ is a parameter such that
$T^\varepsilon<y\leq T^{1/3}$.  For any $T\leq x\leq 2T$, define
\begin{align*}
&{\cal R}_1={\cal R}_1(x,y): =(2\pi)^{-2}x^{9/8}\sum_{n\leq y}\frac{c_n}
 {n^{7/8}}\cos(8\pi\sqrt[4]{xn}-\frac{\pi}{4}),\\
&{\cal R}_2={\cal R}_2(x,y): =\Delta_1(x;\varphi)-{\cal R}_1.
\end{align*}
We shall show that the higher-power moment of ${\cal R}_2$ is small and hence the
integral $\int_T^{2T}\Delta_1^k(x;\varphi)dx$ can be well approximated by
$\int_T^{2T}{\cal R}_1^kdx,$ which is easy to evaluate.

%---------------------------------------------------------------------------
\subsection{\bf Evaluation of the integral $\int_T^{2T}{\cal R}_1^kdx$}
For simplicity we set ${\Bbb I}=\{0,1\}$ and
$$
{\Bbb N}^k=\{\boldsymbol{n}: \boldsymbol{n}=(n_1,\cdots, n_k),n_j\in {\Bbb N},1\leq j\leq k\}.
$$
For each element $\boldsymbol{i}=(i_1,\cdots, i_{k-1})\in {\Bbb I}^{k-1}$,
put $|\boldsymbol{i}|=i_1+\cdots+i_{k-1}$. By the elementary formula
$$
\cos a_1\cdots \cos a_k=\frac{1}{2^{k-1}}\sum_{\boldsymbol{i}\in  {\Bbb I}^{k-1}} \cos
(a_1+(-1)^{i_1}a_2+(-1)^{i_2}a_3+\cdots +(-1)^{i_{k-1}}a_k),
$$
we have
\begin{align*}
{\cal R}_1^k&=( 2\pi)^{-2k}x^{\frac{9k}{8}} \sum_{n_1\leq y}\cdots \sum_{n_k\leq y}
\frac{c_{n_1}\cdots c_{n_k}}{(n_1\cdots n_k )^{7/8}}\prod_{j=1}^k\cos(8\pi\sqrt[4]{n_jx}- \frac{\pi}{4})\\
&=\frac{x^{\frac{9k}{8}}}{( 2\pi)^{2k}2^{k-1}} \sum_{\boldsymbol{i} \in {\Bbb I}^{k-1}}
\sum_{n_1\leq y}\cdots \sum_{n_k\leq y} \frac{c_{n_1}\cdots c_{n_k} }{(n_1\cdots n_k)^{7/8}}
\cos(8\pi\sqrt[4]x\alpha(\boldsymbol{n};\boldsymbol{i})-\frac{\pi}{4} \beta(\boldsymbol{i})),
\end{align*}
where
\begin{align*}
\alpha(\boldsymbol{n};\boldsymbol{i}):& = \sqrt[4]{n_1}+(-1)^{i_1}\sqrt[4]{n_2}
+(-1)^{i_2}\sqrt[4]{n_3}+\cdots +(-1)^{i_{k-1}}\sqrt[4]{n_k},\\
\beta(\boldsymbol{i}): &=1+(-1)^{i_1}+(-1)^{i_2}+\cdots +(-1)^{i_{k-1}}.
\end{align*}
Thus we can write
\begin{equation}
{\cal R}_1^k=\frac{1}{(  2\pi)^{2k}2^{k-1}}(S_1(x)+S_2(x)),
\end{equation}
where
\begin{align*}
S_1(x):&= x^{9k/8}\sum_{\boldsymbol{i}\in {\Bbb I}^{k-1}}\cos\left(-\frac{\pi\beta(\boldsymbol{i})}{4}
         \right) \sum_{\begin{subarray}{c} n_j\leq y,1\leq j\leq k \\
                       \alpha(\boldsymbol{n};\boldsymbol{i})=0\end{subarray}}
       \frac{ c_{n_1}\cdots c_{n_k}}{(n_1\cdots n_k)^{7/8}},\\
S_2(x):&= x^{9k/8}\sum_{\boldsymbol{i}\in  {\Bbb I}^{k-1}}
        \sum_{\begin{subarray}{c} n_j\leq y,1\leq j\leq k \\
              \alpha(\boldsymbol{n};\boldsymbol{i})\not= 0 \end{subarray}}
       \frac{ c_{n_1}\cdots c_{n_k}}{(n_1\cdots n_k)^{7/8}} \cos\left(
        8\pi\alpha(\boldsymbol{n};\boldsymbol{i})\sqrt[4]x-\frac{\pi\beta(\boldsymbol{i})}{4}\right) .
\end{align*}

First consider the contribution of $S_1(x).$ We have
\begin{align}
&\int_T^{2T}S_1(x)dx \\
&=\sum_{\boldsymbol{i} \in {\Bbb I}^{k-1}}
\cos\left(-\frac{\pi\beta(\boldsymbol{i})}{4}\right)\sum_{\begin{subarray}{c}
n_j\leq y,1\leq j\leq k \\ \alpha(\boldsymbol{n};\boldsymbol{i})=0 \end{subarray}}
\frac{ c_{n_1}\cdots c_{n_k}}{(n_1\cdots n_k)^{7/8}}\int_T^{2T}x^{\frac{9k}{8}}dx. \nonumber
\end{align}

It is easily seen that if $\alpha(\boldsymbol{n};\boldsymbol{i})=0,$  then $1\leq
|\boldsymbol{i}|\leq k-1$. Let $l= |\boldsymbol{i}|$, then we have
$$
\sum_{\stackrel{n_j\leq y,1\leq j\leq k}{\alpha(\boldsymbol{n};\boldsymbol{i})= 0}}
\frac{c_{n_1}\cdots c_{n_k} }{(n_1\cdots n_k)^{7/8}}=s_{k;l}(c;y),
$$
where $s_{k;l}(f;y)$ was defined in Lemma 2.4. By Lemma 2.4 we get
\begin{equation}
\int_T^{2T}S_1(x)dx=B_k^{*}(c)\int_T^{2T}x^{\frac{9k}{8}}dx
+O(T^{1+9k/8+\varepsilon}y^{-3/4}),
\end{equation}
where
$$
B_k^{*}(c):=\sum_{\boldsymbol{i} \in  {\Bbb I}^{k-1}}\cos(-\frac{\pi\beta(\boldsymbol{i})}{4})
\sum_{\stackrel{(n_1,\cdots,n_k)\in {\Bbb N}^k}{\alpha(\boldsymbol{n};\boldsymbol{i})= 0}}
\frac{c_{n_1}\cdots c_{n_k} }{(n_1\cdots n_k)^{7/8}}.
$$

For any $\boldsymbol{i}\in {\Bbb I}^{k-1}\setminus \boldsymbol{0}$, let
\begin{equation*}
S(c;\boldsymbol{i}):=\sum_{\stackrel{(n_1,\cdots,n_k)\in {\Bbb
N}^k}{\alpha(\boldsymbol{n};\boldsymbol{i})= 0}} \frac{c_{n_1}\cdots c_{n_k}
}{(n_1\cdots n_k)^{7/8}} .
\end{equation*}
It is easily seen that if $ |\boldsymbol{i}|= |\boldsymbol{i}^{\prime}|$ or $|\boldsymbol{i}|
+|\boldsymbol{i}^{\prime}| =k,$ then
$$
S(c;\boldsymbol{i})=S(c;\boldsymbol{i}^{\prime})=s_{k;|\boldsymbol{i}|}(c).
$$
From  $(-1)^j=1-2j \ \ (j=0,1)$ we also have $\beta(\boldsymbol{i})=k-2|\boldsymbol{i}|$. So we get
\begin{align}
B_k^{*}(c)
&=\sum_{l=1}^{k-1}\sum_{|\boldsymbol{i}| =l} \cos(-\frac{\pi\beta(\boldsymbol{i})}{4})
  S(c;\boldsymbol{i})\\
&=\sum_{l=1}^{k-1}s_{k;l}(c)\cos \frac{\pi(k-2l)}{4}\sum_{|\boldsymbol{i}|=l}1 \nonumber\\
&=\sum_{l=1}^{k-1}{k-1\choose l}s_{k;l}(c)\cos \frac{\pi(k-2l)}{4}=B_k(c).\nonumber
\end{align}

Now we consider the contribution of $S_2(x).$ By the first derivative test and Lemma 2.3 we get
\begin{align}
\int_T^{2T}S_2(x)dx
&\ll T^{\frac 34+\frac{9k}{8}} \sum_{\boldsymbol{i}\in {\Bbb I}^{k-1}}
\sum_{\begin{subarray}{c} n_j\leq y,1\leq j\leq k \\
       \alpha(\boldsymbol{n};\boldsymbol{i})\not= 0\end{subarray}}
       \frac{c_{n_1}\cdots c_{n_k} }{(n_1\cdots n_k)^{7/8}|\alpha(\boldsymbol{n};\boldsymbol{i})|}\\
&\ll T^{\frac 34+\frac{9k}{8}}y^{4^{k-2}-1/4} \sum_{ \boldsymbol{i}\in {\Bbb I}^{k-1}}
\sum_{\begin{subarray}{c} n_j\leq y,1\leq j\leq k \\
      \alpha(\boldsymbol{n};\boldsymbol{i})\not= 0 \end{subarray}}
      \frac{c_{n_1}\cdots c_{n_k} }{(n_1\cdots n_k)^{7/8} }\nonumber\\
&\ll T^{\frac 34+\frac{9k}{8}}y^{4^{k-2}-1/4+k/8}\ll T^{\frac 34+\frac{9k}{8}}y^{b(k) }.\nonumber
\end{align}

Hence from (4.1)-(4.6) we get
\begin{lem} For any fixed $k\geq 3,$ we have
\begin{align}
\int_T^{2T}{\cal R}_1^kdx=\frac{B_k(c)}{(2\pi)^{2k}2^{k-1}}\int_T^{2T} \!\! x^{\frac{9k}{8}}dx
+O(T^{1+\frac{9k}{8}+\varepsilon}y^{-\frac 34}+T^{\frac
34+\frac{9k}{8}+\varepsilon}y^{b(k)}).
\end{align}
\end{lem}

%----------------------------------------------------------------------

\subsection{\bf Higher-power moments of ${\cal R}_2$}
Taking $N=T$ in the formula (2.2) of Lemma 2.1,  we get
\begin{align}
{\cal R}_2 &=(2\pi )^{-2}x^{\frac 98}\sum_{y<n\leq T}\frac{c_n}{n^{7/8}} \cos(8\pi\sqrt[4]{nx}-\pi/4)
+O(T^{1+\varepsilon})\\
&\ll |x^{\frac 98}\sum_{y<n\leq T}\frac{c_n}{n^{7/8}} e(4\sqrt[4]{nx})|+T^{1+\varepsilon},\nonumber
\end{align}
which implies
\begin{align}
&\int_T^{2T}{\cal R}_2^2dx
\ll T^{3+\varepsilon}+\int_T^{2T}|x^{\frac 98}\sum_{y<n\leq T}\frac{c_n}{n^{ 7/8}}
     e(4\sqrt[4]{nx})|^2dt \\
& \quad \ll T^{3+\varepsilon}+T^{13/4}\sum_{y<n\leq T}\frac{c_n^2}{n^{7/4}}
     +T^3\sum_{y<m<n\leq T}\frac{c_nc_m}{(mn)^{7/8}(\sqrt[4]n-\sqrt[4]m)} \nonumber\\
& \quad \ll T^{3+\varepsilon}+\frac{T^{13/4+\varepsilon} } {y^{3/4}} \nonumber \\
& \quad \ll \frac{T^{13/4+\varepsilon}}{y^{3/4}},\nonumber
\end{align}
where we used the estimate (3.3) and the estimate
$$
\sum_{y<m<n\leq T}\frac{c_nc_m }{(mn)^{7/8}(\sqrt[4]n-\sqrt[4]m)} \ll T^{\varepsilon},
$$
which can be proved in a standard way.

Now suppose $y$ satisfies $y^{4b(K_0)}\leq T$. Hence from Lemma 4.1
we get that
$$
\int_T^{2T}|{\cal R}_1|^{K_0}dx\ll T^{1+9K_0/8+\varepsilon},
$$
which implies
\begin{equation}
\int_T^{2T}|{\cal R}_1|^{A_0}dx\ll T^{1+9A_0/8+\varepsilon}
\end{equation}
since $A_0\leq K_0.$ From (1.10) and (4.9) we get
\begin{equation}
\int_T^{2T}|{\cal R}_2|^{A_0}dx\ll \int_T^{2T}(|\Delta_1(x;\varphi)|^{A_0}
+|{\cal R}_1|^{A_0})dx\ll T^{1+9A_0/8+\varepsilon}.
\end{equation}

For any $2<A<A_0,$ from (4.8), (4.10) and  H\"older's inequality we
get that
\begin{align}
\int_T^{2T}|{\cal R}_2|^{A}dx
&=\int_T^{2T}|{\cal R}_2|^{\frac{2(A_0-A)}{A_0-2}+\frac{A_0(A-2)}{A_0-2}}dx \\
& \ll \left(\int_T^{2T}{\cal R}_2^2dx\right)^{\frac{A_0-A}{A_0-2}}
 \left( \int_T^{2T}|{\cal R}_2|^{A_0}dx\right)^{\frac{A-2}{A_0-2}} \nonumber \\[1ex]
& \ll T^{1+\frac{9A}{8}+\varepsilon}y^{-\frac{3(A_0-A)}{4(A_0-2)}}. \nonumber
\end{align}

Therefore we have
\begin{lem} Suppose $T^\varepsilon\leq y\leq T^{1/4b(K_0)},$
$2<A<A_0,$  then
\begin{equation}
 \int_T^{2T}|{\cal R}_2|^{A}dx\ll
T^{1+\frac{9A}{8}+\varepsilon}y^{-\frac{3(A_0-A)}{4(A_0-2)}}.
\end{equation}
\end{lem}

%------------------------------------------------------------------------------
\subsection{\bf Upper bound of the integral $\int_T^{2T}{\cal R}_1^{k-1}{\cal R}_2dx$}
In this subsection we shall estimate the integral $\int_T^{2T}{\cal R}_1^{k-1}{\cal R}_2dx$.
Suppose $T^\varepsilon\leq y\leq T^{1/4b(K_0)},$ which combining Lemma 4.1 implies that
$$
\int_T^{2T}|{\cal R}_1|^{k-1}dx\ll T^{1+\frac{9(k-1)}{8}}.
$$
Thus from (4.7) we get
\begin{equation}
\int_T^{2T}{\cal R}_1^{k-1}{\cal R}_2dx=\int_T^{2T}{\cal
R}_1^{k-1}{\cal R}_2^{*}dx+O(T^{1+\frac{9k}{8}-\frac
18+\varepsilon}),
\end{equation}
where
$$
{\cal R}_2^{*}=(2\pi )^{-2}x^{\frac 98}\sum_{y<n\leq
T}\frac{c_n}{n^{7/8}} \cos(8\pi\sqrt[4]{nx}-\pi/4).
$$

Similarly to (4.1) we can write
$$
{\cal R}_1^{k-1}{\cal R}_2^{*}=\frac{1}{(2\pi)^{2k}2^{k-1}}(S_3(x)+S_4(x)),
$$
where
\begin{align*}
&S_3(x)= x^{\frac{9k}{8}}  \sum_{\boldsymbol{i}\in {\Bbb I}^{k-1} }
\sum_{y<n_1\leq T}\sum_{\stackrel{n_j\leq y,2\leq j\leq k}{\alpha(\boldsymbol{n};\boldsymbol{i})=0}}
\frac{c_{n_1}c_{n_2}\cdots c_{n_k}}{(n_1\cdots n_k)^{7/8}} \cos ( -\frac{\pi}{4} \beta( \boldsymbol{i})),\\
&S_4(x)= x^{\frac{9k}{8}}  \sum_{\boldsymbol{i}\in {\Bbb I}^{k-1}}
\sum_{y<n_1\leq T}\sum_{\stackrel{n_j\leq y,2\leq j\leq k}{\alpha(\boldsymbol{n};\boldsymbol{i})\not=0}}
\frac{c_{n_1}c_{n_2}\cdots c_{n_k} }{(n_1\cdots n_k)^{7/8}} \cos (8\pi \sqrt[4]{x}\alpha(\boldsymbol{n};
\boldsymbol{i} )-\frac{\pi}{4} \beta( \boldsymbol{i})).
\end{align*}

By Lemma 2.4 the contribution of $S_3(x)$ is
\begin{align}
\int_T^{2T}S_3(x)dx
&\ll \sum_{\boldsymbol{i}\in {\Bbb I}^{k-1}}
\sum_{y<n_1\leq T}\sum_{\stackrel{n_j\leq y,2\leq j\leq k}{\alpha(\boldsymbol{n};\boldsymbol{i})=0}}
\frac{c_{n_1}c_{n_2}\cdots c_{n_k}}{(n_1\cdots n_k)^{7/8}} \int_T^{2T}x^{\frac{9k}{8}}dx\\
&\ll \sum_{l=1}^{k-1}|s_{k;l}(c;y)-s_{k;l}(c)|\int_T^{2T}x^{\frac{9k}{8}}dx\nonumber\\[1ex]
&\ll T^{1+\frac{9k}{8}+\varepsilon}y^{-3/4}.\nonumber
\end{align}

By the first derivative test  and Lemma 2.3 we get
\begin{align}
\int_T^{2T}S_4(x)dx
&\ll T^{\frac 34+\frac{9k}{8}}\sum_{\boldsymbol{i}\in {\Bbb I}^{k-1} }
\sum_{y<n_1\leq T}\sum_{\stackrel{n_j\leq y,2\leq j\leq k}{\alpha(\boldsymbol{n};\boldsymbol{i})\not=0}}
\frac{c_{n_1}c_{n_2}\cdots c_{n_k} }{(n_1\cdots n_k)^{7/8}|\alpha(\boldsymbol{n};\boldsymbol{i})|}\\
&\ll T^{\frac 34+\frac{9k}{8}}\sum_{\boldsymbol{i}\in {\Bbb I}^{k-1} }
\sum_{y<n_1\leq k^4y}\sum_{\stackrel{n_j\leq y,2\leq j\leq k}
{\alpha(\boldsymbol{n};\boldsymbol{i})\not=0}} \frac{c_{n_1}c_{n_2}\cdots
c_{n_k} }{(n_1\cdots n_k)^{7/8}|\alpha(\boldsymbol{n};\boldsymbol{i})|}\nonumber\\
& \ \ \ +T^{\frac 34+\frac{9k}{8}}\sum_{\boldsymbol{i}\in {\Bbb I}^{k-1} } \sum_{n_1>k^4y}
\sum_{\stackrel{n_j\leq y,2\leq j\leq k}{\alpha(\boldsymbol{n};\boldsymbol{i})\not=0}}
\frac{c_{n_1}c_{n_2}\cdots c_{n_k} }{(n_1\cdots n_k)^{7/8}n_1^{1/4}}\nonumber\\
&\ll T^{\frac 34+\frac{9k}{8}}y^{b(k)}.\nonumber
\end{align}

From (4.13)--(4.15) we get
\begin{equation}
\int_T^{2T}{\cal R}_1^{k-1}{\cal R}_2dx\ll
T^{1+\frac{9k}{8}+\varepsilon}y^{-3/4}+T^{\frac
34+\frac{9k}{8}}y^{b(k)}+T^{1+\frac{9k}{8}-\frac 18+\varepsilon} .
\end{equation}

%-----------------------------------------------------------------------
\subsection{\bf Proof of Theorem 2.}
Suppose $3\leq k< A_0$ and  $T^\varepsilon\leq y\leq T^{1/4b(K_0)}.$
By the elementary formula $(a+b)^k=a^k+ka^{k-1}b+O(|a^{k-2}b^2|+|b|^k),$ we get
\begin{align}
\int_T^{2T}\Delta_1^k(x;\varphi)dx
&=\int_T^{2T}{\cal R}_1^kdx+k\int_T^{2T}{\cal R}_1^{k-1}{\cal R}_2dx \\
& \quad +O\left(\int_T^{2T}|{\cal R}_1^{k-2}{\cal R}_2^2|dx\right)
        +O\left(\int_T^{2T}|{\cal R}_2|^kdx\right).\nonumber
\end{align}

By (4.9), Lemma 4.2 and H\"older's inequality we get
\begin{align}
&\int_T^{2T}|{\cal R}_1^{k-2}{\cal R}_2^2|dx \\
& \ll \left(\int_T^{2T}|{\cal R}_1|^{A_0}dx\right)^{\frac{k-2}{A_0}}
      \left(\int_T^{2T}|{\cal R}_2|^{\frac{2A_0}{A_0-k+2}}dx\right)^{\frac{A_0-k+2}{A_0}} \nonumber \\
& \ll T^{1+ 9k/8+\varepsilon}y^{-\frac{3(A_0-k)}{4(A_0-2)}}\nonumber.
\end{align}

Now take $y=T^{1/4b(K_0)}.$ From (4.16)-(4.18), Lemma 4.1 and Lemma
4.2 with $A=k$ we get
\begin{align}
\int_T^{2T}\Delta_1^k(x;\varphi)dx&=\frac{B_k(c)}{(\sqrt 2\pi)^k2^{k-1}}\int_T^{2T}x^{9k/8}d x
+O(T^{1+9k/8-\sigma(k,A_0)/4b(K_0)+\varepsilon}) \\
&=\frac{B_k(c)}{(\sqrt 2\pi)^k2^{k-1}}\int_T^{2T}x^{9k/8}dx
+O(T^{1+9k/8-\delta_1(k,A_0)+\varepsilon}). \nonumber
\end{align}
 Theorem 2 follows from (4.19) immediately.

%--------------------------------------------------------------------------
\subsection{\bf Proof of Theorem 3}
Suppose $T^\varepsilon\leq y\leq T^{1/3}.$ Using the arguments of
Theorem 4 directly to ${\cal R}_1$ we get
\begin{equation}
 \int_T^{2T}|{\cal R}_1|^{A_0}dx\ll T^{1+9A_0/8+\varepsilon}
\end{equation}
 holds with $A_0=16/3, T^\varepsilon\leq y\leq T^{1/3}.$
 Thus
\begin{equation}
 \int_T^{2T}|{\cal R}_2|^{A_0}dx\ll T^{1+9A_0/8+\varepsilon}
\end{equation}
 holds with $A_0=16/3, T^\varepsilon\leq y\leq T^{1/3}.$

By the same argument as in last subsection, we get that for $T^\varepsilon\leq y\leq T^{1/3}$,
\begin{equation}
\int_T^{2T}\Delta_1^k(x;\varphi)dx=\int_T^{2T}{\cal R}_1^kdx
+O(T^{1+9k/8+\varepsilon}y^{-\sigma(k,16/3)}+T^{\frac{3}{4}+\frac{9k}{8}}y^{b(k)}).
\end{equation}
Take $y=T^{1/(4b(k)+4\sigma(k,16/3))}.$
From Lemma 4.1 again we get
\begin{align}
\int_T^{2T}\Delta_1^k(x;\varphi)dx
&=\frac{B_k(c)}{(2\pi)^k2^{k-1}}\int_T^{2T}x^{9k/8}d x
  +O(T^{1+9k/8-\frac{\sigma(k,16/3)}{4b(k)+4\sigma(k,16/3)}+\varepsilon})\\
&=\frac{B_k(c)}{(2\pi)^k2^{k-1}}\int_T^{2T}x^{9k/8}dx+
  O(T^{1+9k/8-\delta_2(k,16/3)+\varepsilon}).\nonumber
\end{align}
Now Theorem 3 follows from (4.23).

%==============================================================================

\noindent
Yoshio Tanigawa\\
Graduate School of Mathematics, \\Nagoya University, Chikusa-ku\\
Nagoya 464-8602, Japan\\
E-mail: tanigawa@math.nagoya-u.ac.jp\\

\noindent
Wenguang Zhai and Deyu Zhang\\
School of Mathematical Sciences,\\
Shandong Normal University,\\
Jian, 250014, Shandong, China\\
E-mail: zhaiwg@hotmail.com,  zdy-78@yahoo.com.cn\\

\end{document}